\documentclass[letterpaper,11pt]{amsart}
\usepackage{latexsym,amssymb,amsfonts,amsmath,graphicx,url}
\usepackage[vcentermath,enableskew]{youngtab}
\usepackage{color}
\usepackage{cite}
\usepackage[all]{xy}
\usepackage{verbatim}
\usepackage{tikz,bbm}
\usepackage{caption}
\usepackage{subcaption}
\usepackage{hyperref,cleveref}
\usepackage{mathtools}
\newtheorem{theorem}{Theorem}[section]
\newtheorem{proposition}[theorem]{Proposition}
\newtheorem{lemma}[theorem]{Lemma}

\theoremstyle{definition}
\newtheorem{definition}[theorem]{Definition}

\newtheorem{example}[theorem]{Example}

\theoremstyle{remark}
\newtheorem{remark}[theorem]{Remark}
\numberwithin{equation}{section}

\def\multiset#1#2{\ensuremath{\left(\kern-.3em\left(\genfrac{}{}{0pt}{}{#1}{#2}\right)\kern-.3em\right)}}

\DeclareMathOperator{\supp}{supp}

\newtheorem*{theorem1*}{Theorem \ref{thm:main}}

\definecolor{darkyellow}{HTML}{B9770E}

\newcommand{\sC}{{\sf C}} %clause
\newcommand{\I}{\mathcal{I}} %Intersection
\title{A graph coloring approach to family-based haplotype reconstruction}

\author[Anema]{Jason A.~Anema}
\address{Jason A.~Anema, Division of Statistical Genomics, Department of Genetics, Washington University School of Medicine, St. Louis, USA}
\email{jasona@wustl.edu}

\author[Escobar]{Laura Escobar}
\address{Laura Escobar, Department of Mathematics and Statistics, Washington University in St. Louis, St. Louis, USA}
\email{laurae@wustl.edu}

\thanks{JAA is funded by US National Institute on Aging – NIH Grants U19-AG063893, U01-AG023746, U01-AG023712, U01-AG023749, U01-AG023755, and U01-AG023744.}
\thanks{LE has been partially funded by NSF Grant DMS-1855598 and NSF CAREER Grant DMS-2142656.  }
\rightmark

\begin{document}

\maketitle

\begin{abstract}
Edge Constrained Vertex Coloring (ECVC) problems are defined on a finite multigraph, their solutions are characterized, and a linear time algorithm is given for solving $N$ ECVCs on the same underlying multigraph. Using ECVC problems we develop a novel family-based haplotype reconstruction method which has linear-time complexity in both number of markers and family size and has many other desirable properties. To do so, we define a multigraph given a genomic interval on which a family is recombination-free.
\end{abstract}

\setcounter{tocdepth}{1}
\tableofcontents

\section{Introduction}

Graph coloring problems have a long tradition in mathematics, applied mathematics, operations research, and computer science, see e.g.~\cite{Tuza} for a survey. In this paper we introduce a new graph coloring problem on finite multigraphs which we call an \emph{edge constrained vertex coloring (ECVC) problem} and reduce these class of problems to 2-SAT. As a consequence the worst-case time and space complexity of solving $N$ ECVC problems on a fixed multigraph $G = (V,E)$ is $\mathcal{O}(N|E|)$.

Our motivation to introduce this problem is due to its applications to some classical problems in genomics, particularly in the current era of whole genome sequencing (WGS).  
We now give more details on these applications.

\subsection{Haplotype reconstruction}\label{sec-haplo-rec}
 A haplotype is a sequence $H = (a_1, \ldots, a_n)$ of alleles on a chromosome at $n$ genetic markers $(M_1, \ldots, M_n)$ listed in the correct chromosomal order. When these are markers in an autosomal chromosome of a diploid organism, each individual has two haplotypes $H_1 = (a_1^1, \ldots, a_n^1)$ and $H_2 = (a_1^2, \ldots, a_n^2)$. Current WGS short-read technology outputs genotypes at each genetic marker, i.e. at the $k$-th marker the subject's genotype is a multiset of alleles, $\{ a_k^1, a_k^2\}$. The problem of haplotype reconstruction, also known as {\it phasing}, is to recover the subject's two haplotypes from the genotype data out of the $2^{n-1}$ possibilities. 
 
 Many family-based methods have been developed to approach the haplotype reconstruction problem, for instance see \cite{BB},\cite[Chapter 36]{EZ}, \cite{MM}, and references therein. Two of the most successful algorithms to date are the Elston--Stewart \cite{ES} and Lander--Green \cite{LG}. 
 Despite their usefulness, there are two issues which limit the successful application of both of these algorithms in WGS family data. Both have exponential time complexity, and both are sensitive to genotype errors triggering excess recombination predictions yielding spurious haplotypes. 
 Concretely, Elston--Stewart has exponential complexity in the number of markers considered and Lander--Green is exponential in family size. 
Current WGS data outputs genotypes on tens of millions of markers genome-wide, and many family-based studies have families within their cohort which are much too large for the Lander--Green algorithm to handle. One such study is the Long Life Family Study (LLFS) \cite{Woj2019}, a family-based study of human families enriched for healthy aging phenotypes and exceptional longevity, in which the listed first author is an investigator.

In Section ~\ref{sec:applications} we define a multigraph, called a \emph{family-haplotype multigraph}, and we show how ECVC problems on these multigraphs provide the framework for a novel family-based method to reconstruct haplotypes which, as a consequence of Section~\ref{section:math}, has linear-time complexity in both number of genetic markers and family size. 
The purpose of this work is not to replace the use of Elston--Stewart, Lander--Green, and others, but rather to build upon their strengths and to use ECVC problems to overcome their limitations, thus enabling researchers to reconstruct high-quality haplotypes across entire chromosomes in large pedigrees from WGS genotype data. 

The haplotype reconstruction methods formulated in this work present no limitation on a pedigree's inbreeding status, nor a genetic marker's variant type (SNP, insertion, deletion, copy number variation, et cetera), number of alleles, or allele frequencies (including no assumption of Hardy-Weinberg equilibrium). It allows for missing genotypes and is able to impute subjects' missing haplotypes and genotypes. It is not only robust to genotype errors, but it even detects and flags previously undetectable genotyping errors. This method also provides a way to localize recombination events with great accuracy, and yields no excess recombination prediction, bypassing a major drawback of existing methods. 
 
Our method for solving the haplotype reconstruction problem, as well as some related problems, in families and producing haplotypes of high-quality is made possible by haplotype segments being shared amongst family members. ECVC problems were designed to maximally exploit this shared information.

%%%%%%%
\section{Edge Constrained Vertex Coloring Problems}\label{section:math}

In this section we define a new class of graph coloring problems and give a linear time algorithm to find their solutions. The results of this section provide the framework through which the applications to genomics of Section ~\ref{sec:applications} are made possible.
Throughout this paper, all graphs are finite multigraphs with loops and parallel edges allowed.

Given a set $C$, we denote by \multiset{C}{2} the set of multisets of size two with elements in $C$. We will study the following types of problems:

\begin{definition}\label{def:ECVC} \textit{Edge Constrained Vertex Coloring (ECVC) Problem.} 
Let $G = (V,E)$ be a multigraph, $C$ a set of colors, and an \textit{edge constraint list}  $L: E \rightarrow  \multiset{C}{2}$.
Generate all \emph{vertex colorings} $\phi: V \rightarrow C$ such that for every edge $e \in E$ with endpoints $v,w$ the multiset $\{\phi(v),\phi(w)\}$ equals $L(e)$.
\end{definition}

To ease notation, for $e\in E$ we let $\phi(e)$ denote the multiset $\{\phi(v)\mid v\in e\}$ where $v\in e$ indicates that $v$ is a vertex of $e$.

\begin{remark}
Note that the solutions to ECVC problems are not assumed to be proper, i.e.~solutions $\phi$ are allowed to satisfy $\phi(v)=\phi(v')$ for adjacent $v,v'$. 
\end{remark}

Recall that a vertex of $G$ is isolated if it has degree zero.
We note that if $G$ has an isolated vertex $v$, then a solution $\phi$ to an ECVC problem can take any value in $C$ at $v$. Due to this, as well as the intended applications of ECVC problems in Section~\ref{sec:applications}, we assume throughout this paper that our multigraphs have no isolated vertices.

As we will see below, finding a solution to an ECVC problem (or determining that no solution exists) reduces to a 2-SAT problem. 
To do so we will need to construct some input data from the ECVC problem. 
In fact, this input data will allow us to give an upper bound for the number of solutions to a given ECVC problem.

Let $G$ and $L$ be as in Definition~\ref{def:ECVC}.
Given a subgraph $H$ of $G$, we denote by $L|_H$ the restriction of $L$ to the edges in $H$.
Subgraphs can be used to determine if an ECVC problem has no solution.
In particular, observe that if the ECVC problem on $H$ with constraints $L|_H$ has no solution, then the ECVC problem on $G$ with constraints $L$ also has no solution.

For $v \in V$, let
\begin{equation}\label{eqn:nei}
	\mathcal{I}_v^{G,L}:= \bigcap_{\substack{e\in E \\ v\in e}} \supp(L( e ) ),
	\end{equation}
where $\supp(L(e))$ is the support set of the multiset $L(e)$ (i.e., the set of distinct elements in $L(e)$).
These sets will play a key role in solving ECVC problems.
For example, every solution $\phi$ must satisfy that for all $v\in V$, $\phi(v)\in \I_v^{G,L}$. 

\begin{remark}
The constraint list $L$ cannot be recovered from $\left\{\mathcal{I}_v^{G,L} \mid v\in V\right\}$.
As we can see in Figure~\ref{fig-not-set-coloring}, one can have different constraint lists $L_1,L_2$ such that for all $v\in V$, $\mathcal{I}_v^{G,L_1}=\mathcal{I}_v^{G,L_2}$. 
However, note that the ECVC problem on the left has no solution, whereas the one on the right has a unique solution.
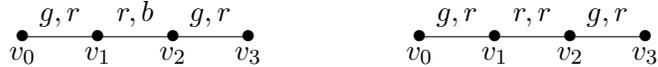
\begin{figure}[h]
\begin{tikzpicture}
\draw (0,0) node {$\bullet$} node[below]{$v_0$}
	--node[above]{$g,r$}	(1,0) node {$\bullet$}node[below]{$v_1$}
	--node[above]{$r,b$}(2,0) node {$\bullet$}node[below]{$v_2$}
	--node[above]{$g,r$}(3,0) node {$\bullet$}node[below]{$v_3$};
\end{tikzpicture}
\qquad\qquad
\begin{tikzpicture}
\draw (0,0) node {$\bullet$} node[below]{$v_0$}
	--node[above]{$g,r$}	(1,0) node {$\bullet$}node[below]{$v_1$}
	--node[above]{$r,r$}(2,0) node {$\bullet$}node[below]{$v_2$}
	--node[above]{$g,r$}(3,0) node {$\bullet$}node[below]{$v_3$};
\end{tikzpicture}
\caption{Two distinct ECVC problems giving rise to the same $\mathcal{I}_v^{G,L}$ at each $v\in V$.}
\label{fig-not-set-coloring}
\end{figure}
\end{remark}

Now, let us discuss the restrictions on ECVC problems imposed by the sets $\I_v^{G,L}$.

\begin{lemma}\label{lem:easy}
Let $L$ be the constraint list of an ECVC problem on a multigraph $G$.
\begin{enumerate}
\item If there exists $v\in V$ such that $\I_v^{G,L}=\varnothing$, then the ECVC problem has no solution.
\end{enumerate}
For $G$ connected we have the following.
\begin{enumerate}\setcounter{enumi}{1}
\item If there exists $v\in V$ such that $|\I_v^{G,L}|=1$, then the ECVC has at most one solution.
\item If $|\I_v^{G,L}|=2$ for all $v\in V$, then $L$ is a constant function with $L(e)$ consisting of two distinct colors, i.e. solutions to the ECVC problem are proper 2-colorings of $G$.
Thus, the ECVC problem has exactly two solutions when $G$ is bipartite and no solution otherwise.
\end{enumerate}
\end{lemma}

\begin{proof}
(1) follows from the observation that if $\phi$ is a solution, then $\phi(v)\in \I_v^{G,L}$ for all $v\in V$.

To prove (2) suppose there exist two solutions $\phi$ and $\psi$. Since $\psi(v),\phi(v)\in \I_v^{G,L}$, $\phi(v)=\psi(v)$.
Given $w\in V$, consider a simple path from $v$ to $w$  with edge sequence $(e_0, e_1, \ldots, e_{n-1})$ and vertex sequence $(v_0, v_1, \ldots , v_n)$.
For all $i$ note that since $L(e_i)=\phi(e_i)$, then $\phi(v_{i+1})$ is the unique element of $L(e_i) \setminus \{\phi(v_{i})\}$.
Since the same is true for $\psi$, we conclude that $\phi(v_i)=\psi(v_i)$ for all $i$ and (2) follows.

Last, we prove (3). 
Note that $\I_v^{G,L} \subset \supp(L(e))$ for all $v\in V$ and $e\in E$ such that $v\in e$.
Since $|\supp(L(e))|\le 2$, we deduce that $\I_v^{G,L}= \supp(L(e))$.
We now show that for any $e,e'\in E$, $L(e)=L(e')$.
Given a pair of edges $e,e'\in E$ let $p\subset G$ be a simple path with edge sequence $(e_0=e, e_1, \ldots, e_{n-1}=e')$ and vertex sequence $(v_0, v_1, \ldots , v_n)$.
For $i=1,\ldots,n-1$, since $v_i\in e_{i-1},e_i$ then $\supp(L(e_{i-1}))=\I^{G,L}_{v_i}= \supp(L(e_{i}))$.
We conclude that $L$ is constant and $|\supp(L(e))|=2$. 
\end{proof}

We make the following proposition explicit, as this observation is useful in the context of ECVC's applications to genomics, in particular to those described in Sections and ~\ref{sec:rf} and ~\ref{sec:ge}.

\begin{proposition}\label{prop-odd}
If $G$ contains an odd cycle $H$ such that the restricted edge constraint list $L|_H$ is a constant function with $L|_H(e)$ consisting of two distinct colors, then the ECVC problem on $G$ has no solution.
\end{proposition}
\begin{proof}
By (3) in Lemma~\ref{lem:easy}, since $H$ is not bipartite, the ECVC problem on $H$ with constraint list $L|_H$ has no solution.
It follows that ECVC problem on $G$ with constraint list $L$ has no solution.
\end{proof}

\begin{example}
Figure~\ref{fig-nosol} shows that the converse of (1) in Lemma~\ref{lem:easy} is not true.
Figure~\ref{fig:caterpillar} illustrates (2) in Lemma~\ref{lem:easy} and 
Figure~\ref{fig:tee} illustrates (3) in Lemma~\ref{lem:easy}.
\begin{figure}[h]
\begin{tikzpicture}
\draw (1,0) node {$\bullet$} 
	--node[above]{$b,g$}(2,0) node {$\bullet$} 
	--node[above]{$g,r$}(3,0) node {$\bullet$}
	--	node[above]{$b,g$}(4,0) node {$\bullet$}
	;
\end{tikzpicture}
\caption{An ECVC problem on a graph with no solution and $\I^{G,L}_{v}\neq\varnothing$ for all $v\in V$.}
\label{fig-nosol}\end{figure}
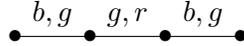

\begin{figure}[h]
\begin{tikzpicture}
\draw (0,0) node {$\bullet$}
	--node[above]{$r,b$}	(1,0) node {$\bullet$}
	--node[above]{$b,g$}(2,0) node {$\bullet$}
	--node[above]{$g,r$}(3,0) node {$\bullet$}
	--node[above]{$g,r$}(4,0) node {$\bullet$}
	(1,0) 	--	node[left]{$b,r$}(1,-1) node {$\bullet$}
	(3,0) 	--	node[right]{$b,r$}(3,-1) node {$\bullet$}
	;
\end{tikzpicture}
\caption{The ECVC problem on this tree and edge constraint list has a unique solution.}
\label{fig:caterpillar}
\end{figure}
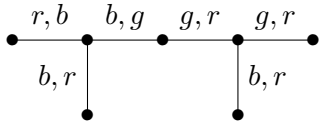

\begin{figure}[h]
\begin{tikzpicture}
\draw (0,0) node {$\bullet$}
	--node[above]{$b,r$}	(1,0) node {$\bullet$}
	--node[above]{$b,r$}(2,0) node {$\bullet$}
	(1,0) node {$\bullet$}
	--	node[left]{$b,r$}(1,-1) node {$\bullet$}
	(1,-1) node {$\bullet$}
	-- node[below]{$b,r$}(2,-1) node {$\bullet$}
	(2,0) node  {$\bullet$}
	-- node[right]{$b,r$}(2,-1) node {$\bullet$}
	(2,0) node  {$\bullet$}
	-- node[above]{$b,r$}(3,0) node {$\bullet$};
\end{tikzpicture}
\caption{The ECVC problem on this graph and edge constraint list has two distinct solutions.}
\label{fig:tee}
\end{figure}
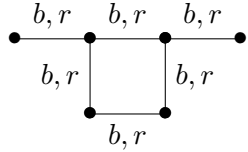
\end{example}

\begin{remark}\label{rem_generate}
A second consequence of Lemma~\ref{lem:easy} is that we can easily generate all solutions to an ECVC problem from an initial solution $\phi$.
Concretely, given a connected component $H$ of $G$ such that $|\I_v^{G,L}|=2$ for all $v$ in $H$ we obtain the second solution by swapping the colors on each vertex.
We now see that if we let $d$ be the number of connected components $H$ of $G$ such that the restricted ECVC problem on $H$ with constraint $L|_H$ has two solutions, then the ECVC problem on $G$ has $2^d$ solutions, all of which can be generated from an initial solution.
\end{remark}

Next, we review 2-SAT problems.

\begin{definition}\label{def:2sat} \textit{2-SAT Problem.} 
Let $x_1,\ldots,x_n$ be Boolean variables and $\sC_1,\ldots,\sC_m$ clauses, each of the form
	\begin{equation*}
	x_i\vee x_j
	,\quad
	x_i\vee \neg x_j
	,\quad\text{or}\quad
	\neg x_i\vee \neg x_j
	\end{equation*}
with $i,j\in\{1, \ldots, n\}$.
Find a truth assignment to the variables that makes the Boolean expression $\sC_1\wedge \cdots \wedge \sC_m$ true, or show that no such truth assignment exists.
\end{definition}

It is well known that the complexity of this problem is $\mathcal{O}(n+m)$.
This follows from observing that finding a solution to a 2-SAT problem or determining that no solution exists is accomplished with 2 successive depth-first searches on the appropriate implication graph, see e.g.~\cite[Lecture 22]{Kozen}.

We now reduce to 2-SAT the problem of finding a solution to an ECVC problem.
Let $G = (V,E)$ and $L: E \rightarrow  \multiset{C}{2}$ as in Definition~\ref{def:ECVC}.
The first step of the reduction is to compute $\I_v^{G,L}$ for all $v\in V$. 
If there exists $v\in V$ such that $\I_v^{G,L}=\varnothing$, Lemma~\ref{lem:easy} immediately tells us that the ECVC problem has no solution and there is no need to invoke a 2-SAT problem.
Therefore, to simplify the presentation of the reduction we will assume that there is no $v\in V$ such that $\I_v^{G,L}=\varnothing$.

Take as Boolean variables $x_c^v$ and $x_c^w$ for $e=(v,w)\in E$ and for each $c\in \supp(L( e))$.
Note that the number of Boolean variables is at most $4|E|$.
Intuitively, the assignment $x_c^v= true$ means ``vertex $v$ is assigned color $c$".
Consider the following formulas:

\begin{itemize}
\item 
For each $e=(v,w)\in E$, let $\alpha_e=\bigwedge_{c\in \supp(L( e))}(x_c^v\vee x_c^w)$. 
Note that $\alpha_e$ is true if and only if each color in $L(e)$ is assigned to a vertex incident to $e$.
\item 
For each $v\in V$, let $\beta_v=\bigvee_{c\in \I_v^{G,L}}x_c^v$.
Note that $\beta_v$ is true if and only if $v$ is assigned at least one color in $\I_v^{G,L}$.
\item
For each $e=(v,w)\in E$ such that $L(e)$ consists of two distinct colors, let $\gamma_e=\bigwedge_{v\in e}(\neg x_c^v\vee \neg x_d^v)$, where $L(e)=\{c,d\}$.
Note that $\gamma_e$ is true if and only if no two distinct colors in $L(e)$ are assigned to a vertex of $e$.
\end{itemize}

Construct a Boolean formula by taking the following conjunction:
	$$
	\Phi=\underbrace{\left(\bigwedge_{e\in E}\alpha_e\right)}_{\sf A}\wedge
	\underbrace{\left(\bigwedge_{v\in V}\beta_v\right)}_{\sf B}\wedge
	\underbrace{\left(\bigwedge_{\substack{e\in E \\ |\supp(L( e))|=2}}\gamma_e\right)}_{\sf C}
	.$$
	
Note that a truth assignment making ${\sf B}\wedge{\sf C}$ true yields the function $\phi: V \rightarrow C$ defined by  
	$$
	\phi(v)=c \ \Longleftrightarrow \ x_c^v=true
	.$$
If this assignment also makes ${\sf A}$ true, then $L(e)=\phi(e)$ for all $e\in E$, i.e.~$\phi$ is a solution to the ECVC problem.
Conversely, if $\phi:V\rightarrow C$ is a solution to the ECVC problem, set $x_c^v=true$ if and only if $\phi(v)=c$. 
The resulting truth assignment makes $\Phi$ true.

\begin{proposition}\label{prop:linear}
The complexity of finding a generating set for the solutions to an ECVC problem, or showing no solution exists, is $\mathcal{O}(|E|)$.
\end{proposition}

\begin{proof}
Computing $\I_v^{G,L}$ for all $v\in V$ can be done with a depth-first search which has complexity $\mathcal{O}(|E|+|V|)$, see e.g.~\cite[Lecture 4]{Kozen}. 
This depth-first search also gives us the connected components of $G$ as well as how many of these components are such that $|\I_v^{G,L}|=2$ for all its vertices.
Since the reduction above has at most $4|E|$ variables and $|V|+ 4|E|$ clauses, the complexity of finding a solution, or showing no solution exists, to the ECVC problem is $\mathcal{O}(|E|+|V|)$.
Remark~\ref{rem_generate} now tells us how to generate all solutions.
Since $G$ was assumed to have no isolated vertices, we have that $|V|\le 2|E|$ and thus $\mathcal{O}(|E|+|V|)=\mathcal{O}(|E|)$.
\end{proof}

\begin{remark}\label{rem:NE}
Fix a multigraph $G = (V,E)$ and a set of colors $C$.
Consider $N$ ECVC problems on $G$, each with edge constraint lists $L_k: E \rightarrow  \multiset{C}{2}$, $k \in \{1, \ldots, N\}$.
These $N$ problems can be solved in linear-time in both $N$ and $|E|$, namely with complexity $\mathcal{O}(N|E|)$.
\end{remark}

\section{Applications of ECVC problems to Genomics}\label{sec:applications}
In this section we use ECVC problems to develop algorithms for various problems in genomics.Most notably we present a solution method to the \emph{haplotype reconstruction problem}, defined in Section~\ref{sec-haplo-rec}, in family data on sexually reproducing diploid organisms (e.g.~humans). Throughout this section, we use the notation introduced in the first paragraph of Section~\ref{sec-haplo-rec}.

Moving forward, we fix a family of individuals with known pedigree (i.e., correct relationships are known), on which a subset of individuals have been sequenced (WGS genotyped). For now we will assume no genotypes are missing and will eliminate this assumption later. Fix a genomic interval of interest on an autosomal chromosome, which contains $n$ genetic markers, $M_1, \ldots, M_n$, on which we will reconstruct the families' haplotypes. Using a small set of high-quality markers, $L_1, \ldots, L_d$, spaced throughout this interval and applying extensions of Elston--Stewart, Lander--Green, or other methods, one can construct the identity-by-descent (IBD) structures, i.e. IBD$_k$ at marker $L_k$ for $k \in \{1, \ldots, d\}$, across the interval of interest with high confidence for the family, assuming a sufficient density of the family has been sequenced. It is common for a family-based study interested in genetics to estimate IBD throughout the genome, for use in linkage analysis, using a relatively small set of high quality markers (called linkage markers).

\begin{figure}[h]
\begin{tikzpicture}
	\node at (-7,0) [circle,draw] (a) {\phantom{$a$}};
	\node at (-6,0) [circle,draw] (1) {};
	\node at (-5,0) [circle,draw] (2) {};
	\node at (-3,0) [circle,draw] (b) {\phantom{$a$}};
	\node at (3,0) [circle,draw] (c) {\phantom{$a$}};
	\node at (5,0) [circle,draw] (3) {};
	\node at (6,0) [circle,draw] (4) {};
        \node at (7,0) [circle,draw] (d) {\phantom{$a$}};
        
        \draw (a) node[below=8] {$L_1$}  -- (1) node[above=8] {$M_1$} -- (2) node[above=8] {$M_2$} -- (-4,0) node[above=8] {$\ldots$} -- (b) node[below=8] {$L_2$} -- (0,0) node[below=8] {$\ldots$} -- (c) node[below=8] {$L_{d-1}$} -- (4,0) node[above=8] {$\ldots$} -- (3) node[above=8] {$M_{n-1}$} -- (4) node[above=8] {$M_n$} -- (d) node[below=8] {$L_d$};
        
\end{tikzpicture}
\caption{High-quality markers $L_1, \ldots, L_d$ spaced throughout markers $M_1, \ldots, M_n$.}
\label{IBDestimate}
\end{figure}
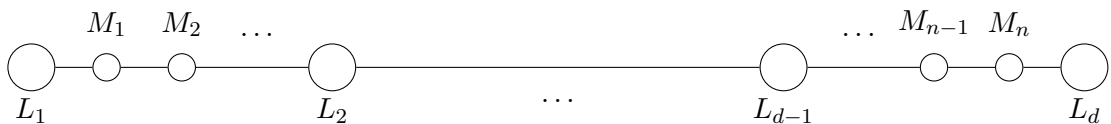

\subsection{ECVC problems for haplotype reconstruction on recombination-free intervals}\label{sec:rf}
Recombination is a natural process that occurs at meiosis between homologous pairs of chromosomes that enables offspring to inherit different combinations of genes than those of their parents, and thus increases genetic variation. For example, on a genomic interval containing markers $M_1, \ldots, M_n$, suppose that a parent has haplotypes $A = (a_1, \ldots , a_n)$ and $B = (b_1, \ldots, b_n)$. If no recombination occurred in this genomic interval during meiosis, then the haplotype an offspring would inherit is either $A$ or $B$.

Within a genomic interval of interest in an autosomal chromosome construct the IBD$_k$ at each marker $L_k$, as above. Identify a subinterval $\I$ on which the family is recombination-free by identifying sequential IBD's which are constant. Provided that markers $L_k$ are relatively close to one another in genetic distance (i.e., in centiMorgan), and in particular due to crossover interference, one can assume that no double-recombination events have occurred between two sequential markers, and thus IBD is constant on the entire interval $\I$ and the family is recombination-free on $\I$. 

Assume $\I$ contains $n$ genetic markers $M_1, \ldots, M_n$. The constant IBD on $\I$ dictates how the haplotypes of the family's founders are inherited throughout the family.
In particular, the IBD determines, for every pair of sequenced individuals, which of their haplotypes (if any) they share. If no such recombination-free interval exist, one can find one by either removing (temporarily) some subjects from the family, or broadening your genomic interval of interest. Figure~\ref{fig_IBD} gives an example pedigree chart depicting the inheritance of nonrecombinant haplotypes consistent with an example IBD for five sequenced family members, labeled 1 through 5.

\begin{figure}[h]
\begin{tikzpicture}
	\node at (-1,0) [circle,draw] (a) {\phantom{$a$}};
        \node at (1,0) [rectangle,draw] (b) {\phantom{$b$}};
        \node at (-2,-2) [circle,draw] (c) {\phantom{$c$}};
        \node at (1,-2) [circle,draw] (1) {$1$};
        \node at (2,-2) [rectangle,draw] (2) {$2$};
        \node at (-3.5,-4) [circle,draw] (3) {$3$};
        \node at (-2.5,-4) [circle,draw] (4) {$4$};
        \node at (-4,-2) [rectangle,draw] (5) {$5$};
        \draw (a) node[left=7] {$A,B$} -- (b) node[right=5] {$C,D$}
        (0,0)--(0,-1)--(-2,-1)--(c)
        (0,-1)--(2,-1)--(2) node[right=7] {$A,C$}
        (1,-1)--(1) node[left=7] {$A,D$}
        (5) node[left=7] {$E,F$}--(c)node[right=6] {$A,D$}
        (-3,-2)--(-3,-3)--(-3.5,-3)--(3) node[left=8] {$A,E$}
        (4) node[right=7] {$E,D$} --(-2.5,-3)--(-3,-3);
\end{tikzpicture}
\caption{Pedigree chart depicting the inheritance of the founders' nonrecombinant haplotypes consistent with an example IBD.}
\label{fig_IBD}
\end{figure}
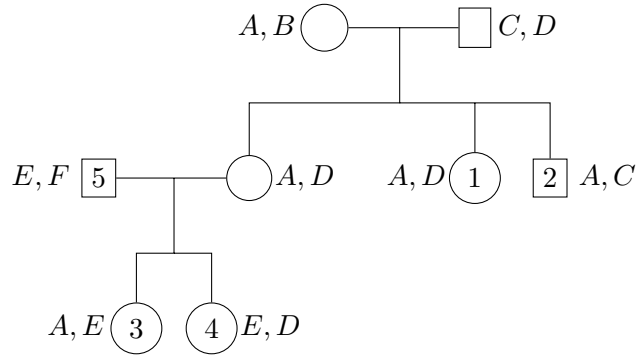

To reconstruct the haplotypes on $\I$ define the \emph{family-haplotype multigraph} to be the multigraph $G$ with vertices all the haplotypes appearing in the sequenced individuals, and an edge for each sequenced individual with vertices the two haplotypes of the individual. In Figure~\ref{fig_fhm} we see the family-haplotype multigraph associated with the IBD structure of Figure~\ref{fig_IBD}. Define the set of colors $C$ by the union of all alleles present in the genotypes of sequenced individuals across the $n$ markers. Define the constraint list, $L_k$, on each edge to be the genotype, at marker $M_k$ of the corresponding individual.

\begin{figure}[h]
\begin{tikzpicture}
	\draw
	(0,0) node {$\bullet$} node[above] {$C$}
	-- node[above] {$2$} (1,0) node {$\bullet$} node[above] {$A$}
	-- node[above] {$1$} (2,0) node {$\bullet$} node[above] {$D$}
	-- node[right] {$4$} (1.5,-1) node {$\bullet$} node[below] {$E$}
	-- node[left] {$3$} (1,0) node {$\bullet$}
	(1.5,-1)
	-- node[below] {$5$} (2.5,-1) node {$\bullet$} node[below] {$F$}
	;	
\end{tikzpicture}
\caption{Family-haplotype multigraph associated with the IBD structure of Figure~\ref{fig_IBD}.}
\label{fig_fhm}
\end{figure}
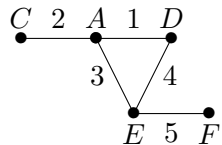

Solve the $n$ ECVC problems on multigraph $G$ with colors $C$ and constraint lists $L_k$, as described in Section~\ref{section:math}.
The solutions, or lack thereof, can be interpreted in the following way. If the $k$-th ECVC problem has a unique solution, then the haplotype reconstruction problem is uniquely determined at marker $k$, with the color assigned to each vertex determining the $k$-th allele of the vertex's corresponding haplotype. If the $k$-th ECVC problem has no solution, then there was a genotyping error (or a mutation) in at least one sequenced individual, and the haplotypes at the $k$-th position cannot be determined. On connected components which have two solutions to the $k$-th ECVC problem, the haplotypes  corresponding to the vertices of this component cannot be determined at the $k$-th marker either. In the two solutions scenario, if the connected component is sufficiently large, this often indicates that the $k$-th genetic marker is unreliable as it displays excess heterozygosity.\\ 
We conclude that on recombination-free intervals with known IBD the haplotype reconstruction problem in families can be solved using ECVC problems in linear-time in both family size, $|E|$, and number of markers, $n$.

\begin{example}
Let us attempt to reconstruct the haplotypes for individuals $1,2,\ldots,5$ of Figure~\ref{fig_IBD} at two genetic markers $M_1$ and $M_2$ using sequencing data in the table below.
\begin{center}
\begin{tabular}{|c|c|c|c|c|c|}
\hline
&$1$ &$2$&$3$&$4$&$5$\\ \hline
$M_1$ & $\{r,r\}$ & $\{r,p\}$ & $\{r,g\}$ &$\{r,g\}$ & $\{r,g\}$
\\ \hline
$M_2$ & $\{r,g\}$ & $\{r,g\}$ & $\{r,g\}$ &$\{r,g\}$ & $\{r,g\}$
\\ \hline
\end{tabular}
\end{center}
\Cref{fig_ex_ecvcs} depicts, from left to right, the ECVC problem for $M_1$, its unique solution, and the ECVC problem for $M_2$.
\begin{figure}[h]
\begin{tikzpicture}
	\draw
	(0,0) node {$\bullet$} %node[above] {$C$}
	-- node[above] {$r,p$} (1,0) node {$\bullet$} %node[above] {$A$}
	-- node[above] {$r,r$} (2,0) node {$\bullet$} %node[above] {$D$}
	-- node[right] {$r,g$} (1.5,-1) node {$\bullet$} %node[below] {$E$}
	-- node[left] {$r,g$} (1,0) node {$\bullet$}
	(1.5,-1)
	-- node[below] {$r,g$} (2.5,-1) node {$\bullet$} %node[below] {$F$}
	;	
\end{tikzpicture}
\hspace{2cm}
\begin{tikzpicture}
	\draw
	(0,0) node {$\bullet$} node[above] {$p$}
	-- %node[above] {$r,p$} 
	(1,0) node {$\bullet$} node[above] {$r$}
	-- %node[above] {$r,r$} 
	(2,0) node {$\bullet$} node[above] {$r$}
	-- %node[right] {$r,g$} (
	(1.5,-1) node {$\bullet$} node[below] {$g$}
	-- %node[left] {$r,g$} 
	(1,0) node {$\bullet$}
	(1.5,-1) -- %node[below] {$r,g$} 
	(2.5,-1) node {$\bullet$} node[below] {$r$}
	;	
\end{tikzpicture}
\hspace{2cm}
\begin{tikzpicture}
	\draw
	(0,0) node {$\bullet$} %node[above] {$C$}
	-- node[above] {$r,g$} (1,0) node {$\bullet$} %node[above] {$A$}
	-- node[above] {$r,g$} (2,0) node {$\bullet$} %node[above] {$D$}
	-- node[right] {$r,g$} (1.5,-1) node {$\bullet$} %node[below] {$E$}
	-- node[left] {$r,g$} (1,0) node {$\bullet$}
	(1.5,-1)
	-- node[below] {$r,g$} (2.5,-1) node {$\bullet$} %node[below] {$F$}
	;	
\end{tikzpicture}
\caption{The leftmost figure is the ECVC problem associated to the sequencing of $M_1$. The figure in the middle is the unique solution to this problem.
The figure on the right is the ECVC corresponding to $M_2$.}
\label{fig_ex_ecvcs}
\end{figure}
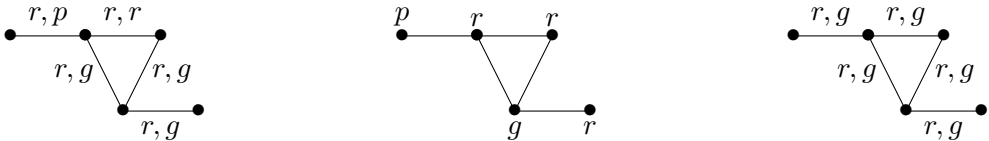

Note that the ECVC problem for $M_2$ has no solution, illustrating \Cref{prop-odd}. This allows us to assert that there is a genotyping error, (or a {\it de novo}  mutation) in at least one of the subjects at $M_2$.

\end{example}

\subsection{ECVC problems and IBD testing, corrections, and discovery.}\label{sec:test} 
On a recombination-free interval, even one subject's graph edge being misspecified (and their descendants inheriting this misspecification) by the IBD estimate can trigger a lot of false ``genotype errors".
In fact, in a sufficiently connected multigraph one misspecified edge can yield ``genotype errors" in as much as $40-50\%$ of heterozygotic markers, greatly exceeding expected genotype error rates of WGS data. 
Sensitivity to misspecified edges in ECVC problems is what empowers the localization of recombination events as discussed in Section ~\ref{sec:rc}. 
Additionally, this sensitivity enables validation of the predicted IBD on recombination-free intervals, correction of misspecified edge(s), the detection of previously undetected recombinations, and allows one to construct the correct IBD (i.e. the appropriate \emph{family-haplotype multigraph}) for a family when other methods are not yielding valid IBD predictions.

\subsection{ECVC problems for recombination localization}\label{sec:rc} 
Suppose we have a genomic interval in an autosomal chromosome containing markers $M_1, \ldots, M_n$ and that a parent has haplotypes $A = (a_1, \ldots , a_n)$ and $B = (b_1, \ldots, b_n)$. In this parent, if exactly one recombination event occurs during meiosis in this genetic interval between markers $M_r$ and $M_{r+1}$, then the haplotype an offspring inherits is either $(a_1, \ldots, a_r, b_{r+1}, \ldots b_n)$ or $(b_1, \ldots, b_r, a_{r+1}, \ldots a_n)$.

Let us assume that on a genomic interval containing markers $M_1, \ldots, M_n$, there is exactly one predicted recombination event, and that at least one of the family members who inherited the resulting recombined haplotype has been WGS genotyped.

Define $IBD_L$ and $IBD_R$ to be the IBD structures at markers $M_1$ and $M_n$, respectively. Assuming this predicted recombination event is real implies that $IBD_L$ and $IBD_R$ are not equal.
Suppose that the recombination occurred between markers $M_r$ and $M_{r+1}$, then $IBD_L$ is the correct IBD on markers $M_1, \ldots, M_r$, and $IBD_R$ is correct on markers $M_{r+1}, \ldots, M_n$. Using $IBD_L$ and $IBD_R$ construct multigraphs $G_L$ and $G_R$, color set $C$ and edge constraint lists $L_k$ for $k = 1, \ldots, n$ as in Section ~\ref{sec:rf}. The edge corresponding to the subject with recombination, as well as the edge of any subject which inherited this recombined haplotype, will have exactly one vertex changed between $G_L$ and $G_R$.

First, solve the $n$ ECVC problems on $G_L$ using the edge constraints $L_k$ for $k \in \{1, \ldots, n\}$ and then solve the ECVC problems on $G_R$ with the same edge constraints. At some index $b$ the $k$-th ECVC problem on $G_L$ will yield no solution at much higher frequency for $k \geq b$, as $G_L$ is incorrectly specified to fit genotypes on markers with index greater that $r$. Similarly, at some index $a$ the $k$-th ECVC problem on $G_R$ will stop yielding no solution at a much higher frequency for $k \leq a$, as $G_R$ is incorrectly specified to fit genotypes on markers with index less than $r+1$. The recombination event is thus most likely to have occurred between markers $M_a$ and $M_b$, thus localizing the recombination from the genomic interval between markers $M_1$ and $M_n$ to a more narrow interval. 

Define $\mathcal{E}$ to be the set with indices $i \in \{a+1, \ldots, b-1\}$ on which the $i$-th ECVC on $G_L$ and the $i$-th ECVC on $G_R$ have solutions (or lack thereof) that differ on the set of vertices shared by $G_L$ and $G_R$. On the markers with indices in $\mathcal{E}$, we cannot solve the haplotype reconstruction problem. 
Fortunately, in real WGS data on large families in LLFS, the size of $\mathcal{E}$ is most often zero. 

When no such indices, $a$ and $b$, are observed, either the predicted recombination is spurious or inconsequential for determining the haplotypes (e.g.~when the predicted recombination only changes edges which are isolated in both  $G_L$ and $G_R$), or differences in $G_L$ and $G_R$ are insufficiently informative to localize the recombination given the genotypes of the family on the $n$ markers. 

When multiple subjects are suspected of having a recombination in an interval of interest, there are a few options. 

\begin{itemize}
\item
When the number of markers is sufficiently large, one could partition the interval, and moving left to right, in increasing marker index order, repeat the above on each subinterval for each potential recombination. 
A similar partitioning strategy can be used when a single subject is suspected of having inherited both a maternal and a paternal recombination. 
\item
When the number of sequenced family members is large, one could focus on resolving one recombination event at a time by removing all other subjects with a predicted recombination or inherited recombination, deleting their corresponding edges from the appropriate multigraphs, and restricting the domain of the edge constraint lists appropriately. 
\item
One could also take a hybrid approach utilizing aspects of both of these strategies. Final localizations of all recombinations should be performed on the appropriate subintervals using all sequenced subjects to obtain the best possible localizations.
\end{itemize}

Recombination localization via ECVC problems together with the ideas of Section ~\ref{sec:rf} enable the reconstruction of haplotypes across entire autosomal chromosomes, and from Proposition ~\ref{prop:linear} it follows that that this can be done with worst-case time  and space complexity $\mathcal{O}(nFR)$, where $n$, $F$, and $R$ are the number of markers on the chromosome, the number of sequenced family members, and the number of recombinations, respectively.

\subsection{ECVC problems and missingness and imputation.}\label{sec:missing} When there are subjects with missing genotypes at the $k$-th marker, delete the corresponding edges from $G$ and solve the ECVC problem on the corresponding submultigraph with the constraint list's domain restricted appropriately. Performing an initial scan of all markers for missingness and recording all missingness patterns to partition the $n$-ECVC problems into problems on a number of submultigraphs still gives worst case performance of $\mathcal{O}(n|E|)$, and can be used in conjunction with all other ideas presented in Section ~\ref{sec:applications}. Imputation of a subject's missing genotypes at marker $k$ is naturally achieved on any subject in which both of the their corresponding edge's vertices are in the vertex set of the $k$-th submultigraph. Imputation of a subject's haplotype(s) at the $k$-th marker can be achieved when at least one of their edge's vertices is in the vertex set of the $k$-th submultigraph. Imputation of haplotypes and genotypes in this way is suitable for both rare and common variants. It is worth noting that for family members on which we have no WGS data (but with known IBD), we can impute their genotypes and haplotype(s) similarly. For example, in the genomic interval $\I$, haplotypes and genotypes of the mother of individuals 3 and 4 in Figure ~\ref{fig_IBD} can be thus imputed. Imputing haplotypes and genotypes increases a study's power by increasing the number of subjects, assuming data on the phenotypes(s) of interest are available for those subjects.

\subsection{ECVC problems and extensions to sex chromosomes.}\label{sec:sexchr}
In humans, and many other sexually reproducing diploid organisms, most individuals have either two copies of the $X$ chromosome, one from each parent, or one copy of chromosome $X$ and one copy of chromosome $Y$. For all individuals, WGS technology returns genotypes, still multisets of size two, for markers on the $X$ chromosome as well as for markers on the $Y$ chromosome. Haplotype reconstruction on the pseudoautosomal regions of the $X$ and $Y$ chromosomes can be performed the same way as on autosomal chromosomes as described above. Haplotype reconstruction on the $Y$ chromosome, outside of the pseudoautosomal regions, is trivial. ECVC problems can be used to reconstruct haplotypes on the nonpseudoautosomal regions of the $X$ chromosome as described above, with one modification: an edge associated to an individual with $XY$ chromosomes in the multigraph should be made a self-loop on the vertex corresponding to their maternally inherited copy of $X$.  Note that for individuals with $XY$ some WGS callers return genotypes on $X$ and $Y$ as multisets of size one. When using such a caller one can merely double the multiplicity of the element in each genotype call on $X$ and on $Y$ and proceed as described above.

\subsection{ECVC problems and their power to detect genotype errors.}\label{sec:ge} 
The ability of ECVC problems to detect genotype errors (or {\it de novo mutations}) on genomic intervals of known IBD is directly related to the connectivity of the underlying multigraph. In particular, it is worth noting that if we fix a spanning forest $T=(V_T, E_T)$ of multigraph $G=(V,E)$, then each edge in $E \setminus E_T$ imposes a stronger constraint to the potential solution(s) to the ECVC problem than edges in $E_T$. For example, as each edge in $E \setminus E_T$ creates a cycle these edges can lead to the ECVC problem not having a solution, such as in Proposition ~\ref{prop-odd}. The number of edges in $E \setminus E_T$ is equal to the first Betti number of the multigraph $G$, showing that the power of an ECVC problem to detect genotype errors is intricately related to the underlying multigraph's first Betti number.

%%%%%%%%
\section{Conclusions}

The methods described in Section ~\ref{sec:applications} are playing a critical role in the Long Life Family Study \cite{Woj2019} to help discover novel complex causal variants driving variation in complex traits related to diabetes, arterial and cardiovascular health, physical function, pulmonary function, cognitive function, and telomere length. One limitation of our methods presented in this work is they require an expert user for successful implementation and can be expensive in human hours. Nonetheless the authors expect that the use of this approach to expand and will be instrumental in resolving variants and haplotypes driving linkage peaks in LLFS, and will greatly enhance the ability for researchers to identify complex genetic structures driving phenotypes in any family-based study with genotype data on sexually reproducing diploid organisms (e.g.~humans).
 
The authors have written a software package that implements all of the ideas in this paper called \emph{HaploGC - Haplotyping via Graph Colorings in Families}. HaploGC has been used in LLFS to construct haplotypes in large multigenerational families (roughly four times larger than the Lander--Green algorithm can handle with current computing power) with runtimes of a few seconds on roughly $100{,}000$ markers. The first listed author announced these new methods in a talk for the {\it 2021 International Genetic Epidemiology Society Meeting} \cite{AEDP}.

\section*{Acknowledgements} We would like to thank Michael A. Province, E.~Warwick Daw, 
Aaron Z.~Palmer, and Peter C.~Samuelson for their helpful conversations and comments.
We also thank anonymous referees for their helpful feedback on the first version of this paper.

\bibliography{refs}
\bibliographystyle{alpha}

\end{document}